\documentclass[11pt,oneside]{amsart}
\usepackage{amscd}
\usepackage{amssymb}
\usepackage{amsfonts}
\usepackage[centertags]{amsmath}
\usepackage{latexsym}
\usepackage{verbatim}
\usepackage{amsthm}
\usepackage{color}
\numberwithin{equation}{section}

\theoremstyle{plain}
\newtheorem{theorem}{Theorem}[section]

\newtheorem{lemma}[theorem]{Lemma}

\theoremstyle{definition}

\newcommand\R{{\mathbb R}}

\newcommand{\g}{\mathfrak{g}}

\begin{document}

\title[Tamed symplectic structures on compact solvmanifolds]{Tamed symplectic structures on  compact solvmanifolds of completely solvable type}
\author{Anna Fino and  Hisashi Kasuya}
\address{(Anna Fino) Dipartimento di Matematica G. Peano \\ Universit\`a di Torino\\
Via Carlo Alberto 10\\
10123 Torino\\ Italy}
\email{annamaria.fino@unito.it}
 \address{(Hisashi Kasuya) Department of Mathematics, Tokyo Institute of 
Technology, 1-12-1, O-okayama, Meguro, Tokyo 152-8551, Japan}
\email{kasuya@math.titech.ac.jp}

\thanks{This work was partially supported by the project PRIN {\em Variet\`a reali e complesse: geometria, topologia e analisi armonica},  the project FIRB {\em Differential Geometry and Geometric functions theory}, GNSAGA (Indam) of Italy and JSPS Research Fellowships for Young Scientists.}

\begin{abstract}
A compact solvmanifold   of completely solvable type, i.e. a compact quotient of a completely solvable Lie group by a lattice,  has a K\"ahler structure if and only if it is a complex torus.  We  show that  a    compact solvmanifold  $M$ of completely solvable type   endowed with an invariant  complex structure $J$ admits  a  symplectic form taming J if and only if   $M$ is a complex torus.  This result generalizes   the   one obtained     in \cite{EFV} for nilmanifolds.  
\end{abstract}

\maketitle

\section{Introduction}

We will say that a symplectic form $\Omega$ 
 on a complex manifold $(M, J)$  tames  the complex structure  $J$ if
$\Omega (X, JX)  > 0$ for  every  non-zero vector field  $X$  on M or, equivalently, if the $(1,1)$-part of  the $2$-form $\Omega$  is positive. By \cite{ST,LZ} a compact complex surface
admitting a  symplectic structure  taming a complex structure is necessarily K\"ahler.  Moreover, by  \cite{Peternell}  non-K\"ahler Moishezon complex structures on compact manifolds can not be tamed by a symplectic form (see also \cite{Zhang}). However, it is still an open problem to find
out an example of a compact  manifold  having a  symplectic structure  taming a complex structure but  not  admitting  any K\"ahler structure. 

Some negative results have been obtained  for compact  nilmanifolds in \cite{EFV} and for special classes of  compact solvmanifolds in \cite{FKV}, where by   a compact nilmanifold  (respectively solvmanifold) we mean  a compact quotient of a nilpotent (resp. solvable) Lie group  $G$ by a discrete subgroup $\Gamma$.  

We recall that   a compact  nilmanifold  is K\"ahler  if and only if it is diffeomorphic to  a torus (\cite{BG, Hasegawa}).   Benson and Gordon  in \cite{BG}  conjectured that    if  a  compact solvmanifold $\Gamma \backslash G$ of  completely solvable type  admits a K\"ahler
metric then $M$  is diffeomorphic to a standard $2n$-torus.   
The conjecture was proved  by  Hasegawa \cite{Ha2}  and  more  generally it was showed that a compact solvmanifold is K\"ahler if and only if it a finite quotient of a complex torus which has the structure of a complex torus bundle over a complex torus.  A similar result  was proved by Baues and Cortes in \cite{BC2}  for  K\"ahler infra-solvmanifolds,  showing the relation   of the Benson-Gordon conjecture  to the more general problem of aspherical K\"ahler manifolds with solvable fundamental group.

In \cite{EFV}   it has been shown that a  compact nilmanifold $\Gamma \backslash G$  endowed with  an invariant complex structure $J$, i.e.  a complex structure which comes from a left invariant complex structure on $G$, admits a symplectic form taming $J$ if and only if $\Gamma \backslash G$ is a torus.  The result was obtained by using a characterization of compact nilmanifolds admitting  pluriclosed metrics, i.e. Hermitian metrics such that its fundamental form $\omega$ satisfies the condition $\partial \overline \partial \omega =0$.  

For compact solvmanifolds by \cite{FKV}  if   $J$ is invariant under the action of a nilpotent complement of the nilradical of  $G$, $J$ is abelian or $G$ is almost abelian (not of type (I)), then the  compact solvmanifold $\Gamma \backslash G$ cannot admit any symplectic form taming the complex structure $J$, unless  $\Gamma \backslash G$ is K\"ahler. 
For a   compact solvmanifold  $\Gamma \backslash G$   of completely solvable type,  no general result  about  the existence of  symplectic  forms  taming   complex structures is  known.  In the present paper we will show that if such  symplectic structures  exist then the compact solvmanifold has to be a torus and therefore  a K\"ahler manifold.

We note that at  the level of Lie algebras,  a  unimodular completely  solvable K\"ahler Lie
algebra  is necessarily abelian.  As remarked in \cite{BC2} this follows from  Hano's result \cite{Hano}  that a   unimodular K\"ahler Lie  group   has to be 
flat  and  from the classification
of flat Lie groups obtained in \cite{Milnor}. 
K\"ahler Lie algebras have been also studied in   \cite{DM}, where     the so-called     K\"ahler double extension has been introduced.  The K\"ahler double extension realizes a K\"ahler Lie algebra as the K\"ahler reduction of another one.

In \cite{BC} it is  shown  that every symplectic
Lie group admits a sequence of  symplectic reductions to
a unique irreducible symplectic Lie group. In particular,  every symplectic
 completely solvable Lie group  is always symplectically reduced. By using this symplectic reduction we show that  if $\frak g$ is    unimodular   completely solvable Lie algebra admitting a symplectic form $\Omega$ taming a complex structure $J$, then $\frak g$  has to be abelian. 
 As a consequence we prove  the following

 {\bf Theorem}    {\it A  compact solvmanifold  $M$ of completely solvable type   endowed with an invariant  complex structure $J$ admits  a  symplectic form taming $J$ if and only if   $M$ is a complex torus. }

\medskip

{\it Acknowledgements.}  The authors are grateful to Vicente Cortes and the referee  for useful comments that have helped us to improve the final version of the paper.

\section{Preliminaries}

A Lie algebra $\frak g$ is called completely solvable   if $ad_X: \frak g \rightarrow \frak g$ has only real eigenvalues for every $X \in \frak g$. Equivalently, $\frak g$ is isomorphic to a subalgebra of the (real) upper triangular matrices in $\frak{gl} (m, \R)$ for some $m$. In particular,
nilpotent Lie groups are completely solvable.

 Lattices in  completely solvable  Lie groups satisfy strong rigidity properties,  which are similar 
to the Malcev rigidity  \cite{Malcev} of lattices in nilpotent Lie groups. These properties  were first
showed  by Saito in  \cite{Saito}.

A  unimodular K\"ahler Lie  group   has to be 
flat, as proved  by Hano in \cite{Hano}.  Therefore a  completely solvable unimodular
K\"ahler Lie  group   is abelian   by virtue of Hano's result and the classification
of flat Lie groups obtained in \cite{Milnor}.  For  non-unimodular  K\"ahler Lie algebras general results have been obtained by Gindikin, Vinberg, Pyatetskii-Shapiro in  \cite{GVP} (see also \cite{Do}).

In \cite{BC} it is  shown  that every symplectic
Lie group admits a sequence of subsequent symplectic reductions to
a unique irreducible symplectic Lie group.

Let  $(\frak g, \Omega)$  be a symplectic Lie algebra. An ideal $\frak h$  of  $\frak g$ is called an isotropic
ideal of  $(\frak g, \Omega)$  if  $\frak h$  is an isotropic subspace for  $\Omega$, i.e. $\Omega \vert_{\frak h \times \frak h} =0$.  Note that $\frak h$ is isotropic if and only if $\frak h \subseteq \frak h^{\perp_{\Omega}}$, where  $\frak h^{\perp_{\Omega}}$   is the orthogonal complement of $\frak h$ in $\frak g$ with respect to $\Omega$.

By Lemma  2.1 in \cite{BC} one has the following properties

\begin{enumerate}
\item  If  $\frak h$ is  an isotropic  ideal of  a  symplectic Lie algebra $(\frak g, \Omega)$ then $\frak h$ is abelian.

\item If  $\frak h$ is  an  ideal of  a  symplectic Lie algebra $(\frak g, \Omega)$, then $\frak h^{\perp_{\Omega}}$ is a Lie subalgebra of $\frak g$.

\item  If  $\frak h$ is  an  ideal of  a  symplectic Lie algebra $(\frak g, \Omega)$, then  $\frak h^{\perp_{\Omega}}$  is an ideal in $\frak g$  if and only if  $[\frak h^{\perp_{\Omega}}, \frak h ] = 0$.

\end{enumerate}

We will now review briefly the symplectic reduction. If   $(\frak g, \Omega)$ is a symplectic Lie algebra and $\frak h \subseteq \frak g$  is  an isotropic ideal, then   $\frak h^{\perp_{\Omega}}$ with respect to $\Omega$  is a  Lie subalgebra of $\frak g$ containing $\frak h$
 and therefore  $\Omega$  descends
to a symplectic form $\tilde \Omega$ on the quotient Lie algebra $\frak h^{\perp_{\Omega}}/\frak h$.

The symplectic Lie algebra $(\frak h^{\perp_{\Omega}}/\frak h, \tilde \Omega)$  is called the symplectic
reduction of  $(\frak g, \Omega)$ with respect to the isotropic ideal $\frak h$.

If  $\frak g$ is completely solvable then by \cite[Example 2.4] {BC}, 
  $\frak g$ contains a nontrivial  ideal  $\frak h$ which is isotropic.

In the case of symplectic  forms taming complex structures  we now  show  the following

\begin{lemma} Let $\frak g$ be  a completely solvable Lie algebra endowed with a symplectic form $\Omega$ taming a complex structure $J$.  Then one has the  following decomposition  (as vector spaces):
$$
\frak g = J \frak h  \oplus \frak h^{\perp_{\Omega}},
$$
with $\frak h$ an  abelian   isotropic ideal of $\frak g$ and $ \frak h^{\perp_{\Omega}}$ a Lie subalgebra of $\frak g$.
\end{lemma}

\begin{proof} Since $\dim J \frak h = \dim \frak h = \dim \frak g - \dim \frak h^{\perp_{\Omega}}$, it is sufficient to show that $J \frak h \cap \frak h^{\perp_{\Omega}} = \{ 0 \}$. Let $Y$  be a non zero element  belonging to the intersection $J \frak h \cap \frak h^{\perp_{\Omega}}$. Then one has that $Y = JX$, with $X \in \frak h$ and $X \neq 0$. But then since $Y \in \frak h^{\perp_{\Omega}}$, one should have in particular 
$$
\Omega  (JX, X) =0,
$$
which is not possible since $\Omega$ tames $J$. 
\end{proof}

The ideal $\frak h$ is at least $1$-dimensional, so we can suppose that $\dim  \frak h = 1$.

\begin{lemma} \label{lemmaquotient} Let $\frak g$ be  a completely solvable Lie algebra endowed with a symplectic form $\Omega$ taming a complex structure $J$ and $\frak h$  be a $1$-dimensional isotropic ideal.   Then   the Lie algebra $ \frak h^{\perp_{\Omega}}/\frak h$ admits a symplectic form $ \tilde \Omega$ taming a complex structure $ \tilde J$.

If  $\frak g$ is unimodular and   $\frak h^{\perp_{\Omega}}$ is an  abelian ideal of $\frak g$, then $\frak g$ is abelian. \end{lemma}

\begin{proof} Suppose that $\frak h = span < X>$. So $J \frak h = span < JX>$ and we know that, since $\frak h$ is an ideal then $[X, \frak g] \in span <X>$.
We have that 
 $$
 J \left ( Y  - \frac{\Omega (JY, X)}{\Omega (JX, X)} X  \right )  \in  \frak h^{\perp_{\Omega}},  
 $$
for every   $Y   \in \frak h^{\perp_{\Omega}}$. Moreover, since $\frak h$ is an ideal,  the complex structure $J$ induces a complex structure $\tilde J$ on $\frak h^{\perp_{\Omega}}/ {\frak h}$, defined by
$$
\tilde J (Y + \frak h) =  JY   + \frak h.
$$
Indeed, if $JY$ does not belong to $\frak h^{\perp_{\Omega}}$, one changes $Y + \frak h$ to  $$Y  - \frac{\Omega (JY, X)}{\Omega (JX, X)} X + \frak h.$$
Moreover,  $\frak h^{\perp_{\Omega}}/{\frak h}$ has  a symplectic structure $\tilde \Omega$ induced by $\Omega$.  Moreover,  the complex structure on  $\frak h^{\perp_{\Omega}}/{\frak h}$ is tamed by $\tilde \Omega$.

Suppose   that $\frak h^{\perp_{\Omega}}$ is an abelian  ideal of   $\frak g$.  Then  $\frak g =  J  \frak h \ltimes \frak h^{\perp_{\Omega}}$ and it is almost abelian.  Since $\frak g$ is not of type (I), by Proposition 7.1 in \cite{FKV} we have that $\frak g$ has to be abelian.
\end{proof}

\section{Main result}

By \cite{EF} a $4$-dimensional  completely solvable Lie algebra endowed with a symplectic form $\Omega$ taming a complex structure $J$ is necessarily K\"ahler and if it is unimodular, then it is abelian. We now show  that in every dimension a  unimodular   completely solvable Lie algebra endowed with a symplectic form $\Omega$ taming a complex structure $J$ is  abelian.

\begin{theorem} Let $\frak g$ be  a  $2n$-dimensional    unimodular completely solvable Lie algebra endowed with a symplectic form $\Omega$ taming a complex structure $J$, then $\frak g$  is abelian.

\end{theorem} \begin{proof} For $n = 2$ we know   by \cite{EF}  that  the theorem holds.  We will prove the theorem by induction. 

  Let $\frak h =  {\mbox {span}}  <X>$  be a $1$-dimensional isotropic ideal $\frak h$ of $\frak g$.
Since  $ \dim \frak g = 2n$,  by  the previous Lemma we can choose a basis 
  $
   \{ X, JX, Y_1, JY_1,  \ldots, Y_{n -1},  J Y_{n -1} \}
  $
  of $\frak g$ with  $Y_l,  JY_l  \in \frak h^{\perp_{\Omega}}$, $l = 1, \ldots, n - 1$.   
  By  Lemma  \ref{lemmaquotient}  the $(2n-2)$-dimensional  Lie algebra $\frak h^{\perp_{\Omega}}/\frak h$, which can be identified with  $\frak v = span <Y_1, JY_1,  \ldots, Y_{n -1},  J Y_{n -1}>$, has a tamed complex structure. 
  
For  every $Y \in \frak v$ we have $\Omega (Y, X) =0 = \Omega (JY, X)$ and 
 $$
 \begin{array}{l}
 [X, Y] = a_1 X, \quad  [X, JY] = a_2 X,  \quad [JX, Y] = b_1 X + b_2 JX + Z_1,\\[3pt]
 [JX, JY] = c_1 X + c_2 JX + Z_2,
 \end{array}
 $$
 with $\Omega (Z_i, X) = \Omega (J Z_i, X) = 0$,  $i = 1,2$. 
 
 By the integrability of $J$ 
 $$
 [JX, JY] = [X,Y] + J [JX, Y] + J [X, JY]
 $$
 we obtain 
 $$
 (c_1 + b_2 - a_1) X + (c_2 - b_1 - a_2) JX + Z_2 - J Z_1 =0
 $$
  and therefore 
  $$
  (c_2 - b_1 - a_2) \Omega (JX, X) =0.
  $$
  Since $\Omega(JX, X) \neq 0$, we get $$c_2 = b_1 + a_2,   \quad Z_2 = JZ_1 + (a_1 - b_2 - c_1) X.$$
  
  As a consequence
  $$
  [JX, JY] = (a_1 - b_2) X + (b_1 + a_2) JX + JZ_1.
  $$
  Since $d \Omega =0$, we have
  $$
  \Omega ([JX, Y], X) = - \Omega ([Y, X], JX) - \Omega ([X, JX], Y)
  $$
  and thus
  $$
  b_2 \Omega(JX, X) = - a_1 \Omega (JX, X)
  $$
  which implies that $b_2 =  - a_1$.
  By the condition
  $$
  \Omega([JX, JY], X) = - \Omega([JY, X], JX) - \Omega([X, JX], JY)
  $$
  we get
  $$
  -  (b_1 + a_2) \Omega (X, JX) = a_2 \Omega(X, JX)
  $$
  and therefore $b_1 = - 2 a_2$.
  
  Then, for every $Y \in \frak v$,   we have
\begin{equation}\label{bracketX}
  \begin{array}{l}
  [X, Y] = a X, \quad [X, JY] = b X, \quad [JX, Y] = - 2 b X -  a JX + Z_1,\\[3pt]
  [JX, JY] = 2 a X - b JX + JZ_1,
  \end{array}
\end{equation}
  with
  $$
  a = \frac{\Omega ([X, Y], JX)}{\Omega(X, JX)}, \quad b = \frac{\Omega ([X, JY], JX)}{\Omega(X, JX)}.
  $$
Moreover by using that $\Omega(Z_1, X) =0 = \Omega (JZ_1, X)$ we obtain that 
$$
Z_1  \in span<Y_1, J Y_1, \ldots, Y_{n - 1}, J Y_{n -1}> = \frak v.
$$
 By  \eqref{bracketX}  and using  the basis $\{ X, JX, Y_1, J Y_1, \ldots, Y_{n -1}, J Y_{n -1} \}$ of $\frak g$   we have, that for every $Y \in \frak v$ 
$$
trace (ad_Y) =  (- a) + (a) + trace (ad_Y \vert_{\frak v}) =0,
$$
since $\frak g$ is unimodular, where  $ad_Y \vert_{\frak v}$ denotes  the $(2n-2) \times (2n-2)$ block  matrix of  $ad_Y \vert_{\frak  v}$with respect to the basis $Y_1, J Y_1, \ldots, Y_{n -1}, J Y_{n -1}$.
This implies that $trace (ad_Y \vert_{\frak v }) =0$, for every $Y \in  \frak v$. Note that by using the identification  $\frak v \cong \frak h^{\perp_{\Omega}}/\frak h$ as vector spaces we have that $$[Y, Z] + {\frak h}  = [Y, Z]_{\frak v} + \frak h, \quad \forall Y, Z \in \frak v$$
where by  $ [Y, Z]_{\frak v}$ we denote the component of $[Y, Z]$ on $\frak v$.
This implies that $\frak h^{\perp_{\Omega}}/\frak h$  has to be unimodular.

 Therefore by induction we can suppose that  the symplectic reduction $\frak h^{\perp_{\Omega}}/\frak h$ has  a K\"ahler structure. Since  $\frak h^{\perp_{\Omega}}/\frak h$ is completely solvable and unimodular, it has to be abelian. So as a vector space we have
  $$
  \frak g = span <X, JX> \oplus \,  \frak v,
  $$
  with  
  $span  <X> \oplus  \,  \frak  v$ a Lie subalgebra of $\frak g$ and
 $$
 span<Y_1, JY_1,  \ldots, Y_{n -1},  J Y_{n -1}> =  { \frak v}
 $$
    $J$-invariant    and such that $[\frak v, \frak v] \subseteq span <X>.$ Suppose that ${\frak h}^{\perp_{\Omega}}$ is not an ideal of $\frak g$, then  there exists a non-zero  $Y \in \frak v$,  such that $[JX, Y] \notin {\frak h}^{\perp_{\Omega}}$. Thus
 \begin{equation}\label{bracketswithperp}
[X, Y] = a X, \,  [X, JY] = b X, \, [JX, Y] = - 2 b X - a JX + Z_1, \,  [JX, JY] = 2 a X - b JX + J Z_1
\end{equation}
 with $a \neq 0$ and $Z_1 \in \frak v$.  By the Jacobi identity we obtain
 $$
 [[JX, Y], JY] + [[Y, JY], JX] + [[JY, JX],Y] =0
 $$
 and therefore in particular that the component in $\frak v$ has to vanish
 $$
 - a J Z_1 + b Z_1 =0
 $$
and so   $Z_1 =0$. If we put $[X, JX] = h X$, by the Jacobi identity
 $$
 [[X, JX], Y] + [[Y, X], JX] + [[JX, Y], X] = a h X=0
 $$
 we get $h=0$ and so $[X, JX] =0$.
 By using again  the Jacobi identity
 $$
 [[JX, Y], JY] + [[JY, JX], Y] + [[Y, JY], JX] = (- 4 b^2 - 4 a^2) X =0
 $$
 we get a contradiction. So ${\frak h}^{\perp_{\Omega}}$ has to be an ideal of $\frak g$. Moreover, with the same argument as before  we can show that $[X, Y] =0$, for every $Y \in {\frak v}$. Indeed, if there exists a non-zero $Y$ such that $[X, Y] \neq 0$, then $[JX, Y]$ has a non zero component along $JX$.
 Therefore,  for every $Y \in \frak v$, we have  $[JX, Y] \in \frak v$.
 
Thus, we have   the relations \eqref{bracketswithperp} with $a = b =0$, i.e. 
$$
[X, Y] =0, \, [X, JY] =0, \, [JX, Y] = Z_1,  \, [JX, JY] = J Z_1,
$$
for every $Y \in \frak v$ and so in particular 
  $[JX, JY] = J [JX, Y]$. Therefore \begin{equation}\label{commuting} ad_{JX} \circ J (Y) = J \circ ad_{JX} (Y),  \quad \forall Y \in {\frak v}. \end{equation}

Since  $\frak h^{\perp_{\Omega}} = span <X> \oplus  \, \frak v$ is nilpotent,  $\frak h^{\perp_{\Omega}}$ coincides with  the nilradical  $\frak n$ of $\frak g$. Indeed $[\frak h^{\perp_{\Omega}}, \frak h^{\perp_{\Omega}}] \subseteq \frak h$ and $[\frak h, \frak h^{\perp_{\Omega}}] = \{ 0 \}$.

Note that  if $[X, JX] =0$,  then we have that $span<JX>$ is a nilpotent complement of $\frak n$ such that $J ad_{JX} = ad_{JX} J$ and we can apply  Theorem 1.2 in \cite{FKV} to conclude that $\frak g$ has to be abelian.

Now in order to complete the proof we claim  that if $h \neq 0$  then we get a contradiction. Suppose that $[JX, X]  = h X$, with $h \neq 0$. If $\frak h^{\perp_{\Omega}}$ is abelian, then  $\g=\R JX\ltimes \R^{2n-1}$ and by \cite{FKV} an almost abelian completely solvable Lie algebra  admitting a symplectic form taming a complex structure has to be abelian.

Since $ad_{JX}$ is unimodular on $\frak h^{\perp_{\Omega}}$,
$ad_{JX}$ has  an eigenvalue $k$ such that $sign(h)\not=sign(k)$.
Take  the generalized eigenspace (i.e. the eigenspace of the semi-simple part $(ad_{JX})_{s}$  of $ad_{JX}$)  $V\subset \frak h^{\perp_{\Omega}}$ for the eigenvalue $k$.
Since  $ad_{JX}( {\frak v})\subset  {\frak v}$, $ad_{JX}(X)=hX$ and $h\not=k$, we have  $V\subset  {\frak v}$.
Then   we get  $[JX,V]\subset V$  and  by \eqref{commuting}     $ad_{JX} J (Y)  = J ad_{JX} (Y)$, for every $Y \in V$. 
For $Y_{1},Y_{2}\in V$, we have $[Y_{1},Y_{2}]=cX$ for some $c$.
Since the semi-simple part $(ad_{JX})_{s}$ of $ad_{JX}$ is a derivation on $\frak h^{\perp_{\Omega}}$, 
we get 
\[2kcX=hcX.
\]
By $sign(h)\not=sign(k)$, we obtain  $c=0$ and hence $[V,V]=0$.
Take a subspace $W\subset \frak h^{\perp_{\Omega}}$ such that
 \[\frak h^{\perp_{\Omega}}=V\oplus W \oplus span< X>\] and  
\[[JX, W \oplus span< X>] \subset  W \oplus span< X>\]
by using the generalized eigenspace decomposition.
We have 
$[V,W]\subset span< X>$ and $[W,W]\subset span< X>$.
Consider 
\[\bigwedge \g^{\ast}=\bigwedge  ( span< Jx>\oplus V^{\ast}\oplus W^{\ast} \oplus span< x>).
\]
Then we have $dV^{\ast}\subset Jx\wedge V^{\ast}$, $dW^{\ast} \subset  Jx\wedge W^{\ast}$ and 
\[dx\in V^{\ast}\wedge W^{\ast} \oplus Jx\wedge W^{\ast}\oplus W^{\ast}\wedge W^{\ast}\oplus span<Jx\wedge x>
\]
by $[JX,V]\subset V$ and  $[V,V]=0$.
We obtain
\[d(Jx\wedge V^{\ast})=0,
\]
\[d(Jx\wedge W^{\ast})=0,
\]
\[d(Jx\wedge x)\subset jx\wedge V^{\ast}\wedge W^{\ast}\oplus Jx\wedge W^{\ast}\wedge W^{\ast},
\]
\[d(V^{\ast}\wedge V^{\ast})\subset Jx\wedge V^{\ast}\wedge V^{\ast},
\]
\[d(V^{\ast}\wedge W^{\ast})\subset Jx\wedge V^{\ast}\wedge W^{\ast},
\]
\[d(V^{\ast}\wedge x)\subset Jx\wedge V^{\ast}\wedge x\oplus V^{\ast}\wedge V^{\ast}\wedge W^{\ast}\oplus Jx\wedge V^{\ast}\wedge W^{\ast}\oplus V^{\ast}\wedge V^{\ast}\wedge W^{\ast},
\]
\[d(W^{\ast}\wedge W^{\ast})\subset Jx\wedge W^{\ast}\wedge W^{\ast}
\]
and 
\[d(W^{\ast}\wedge x)\subset Jx\wedge W^{\ast}\wedge x\oplus V^{\ast}\wedge W^{\ast}\wedge W^{\ast}\oplus W^{\ast}\wedge W^{\ast}\wedge W^{\ast}\oplus x\wedge W^{\ast}\wedge W^{\ast} \oplus  Jx\wedge W^*\wedge W^*.\]
Consider a generic   $2$-form  $\omega \in \bigwedge^{2} \g^{\ast}$. We can write $\omega$ as  $$\omega =\omega_{1}+\omega_{2}$$  with  $\omega_{1}\in V^{\ast}\wedge V^{\ast}$ and $\omega_{2}$ belonging  to the  complement of $V^{\ast}\wedge V^{\ast}$.

If we impose $d\omega=0$,  then $d\omega_{1}=0$.
Since $ad_{JX}(V)\subset V$ and $ad_{JX}\circ J=J\circ ad_{JX}$, considering the triangulation of $ad_{JX}$ on $V$,
we can take a basis  $(Y_{1}, Y_{2},\dots, Y_{2s-1},Y_{2s})$  of  $V$  such that 
\[
 [JX, Y_{i}]=kY_{i} \,\,\, mod\,\,  span<  Y_{1}, \dots, Y_{i-1}>.
\]
and 
$JY_{2i}=Y_{2i-1}$.
Consider the dual basis  $(y_{1}, y_{2},\dots, y_{2s-1},y_{2s})$ of  $(Y_{1}, Y_{2},\dots, Y_{2s-1},Y_{2s})$.
Then 
\[d(y_{2i}\wedge y_{2i-1})=-2k Jx\wedge  y_{2i} \wedge y_{2i-1} \,\, \, mod \,\, Jx\wedge  span_{j_{1}+j_{2}<4i-1}< y_{j_{1}}\wedge y_{j_{2}}>.
\]
If $\omega_{1}$ has a non-trivial $y_{2s}\wedge Jy_{2s}$-component, then we have $d\omega_{1}\not=0$.
Hence we get  $$\omega(Y_{2s},JY_{2s})=\omega_{1}(Y_{2s},Y_{2s-1})=0$$  for  every closed $2$-form $\omega \in \bigwedge^{2} \g^{\ast}$.
This implies that  any  closed $2$-form $\omega \in \bigwedge^{2} \g^{\ast}$ cannot be a symplectic form taming $J$.
Thus $h\not=0$ is impossible.

\end{proof}

As a consequence of previous theorem  we can prove  the following
\begin{theorem}
A  compact solvmanifold  $M= \Gamma \backslash G$ of completely solvable type   endowed with an invariant  complex structure $J$ admits  a  symplectic form taming $J$ if and only if   $M$ is a complex torus. 
\end{theorem}

\begin{proof}  Since $G$ admits a compact quotient by  a lattice,  by Milnor's result  \cite{Milnor} the Lie group $G$ has to be unimodular.   Suppose that $M$ admits a symplectic structure $\Omega$ taming an invariant  complex structure  $J$ on $M$, by using  symmetrization process we can suppose that $\Omega$ is invariant (see \cite{EFV}). So the Lie algebra of $G$ has to be a  unimodular completely solvable Lie algebra admitting a  symplectic form taming a complex structure.  Therefore,  by previous theorem   $\frak g$ has to be abelian.
\end{proof}

\end{document}